\documentclass[a4paper,english,fontsize=10pt,parskip=half,abstracton]{scrartcl}
\usepackage{babel}
\usepackage[utf8]{inputenc}
\usepackage[T1]{fontenc}
\usepackage[a4paper,left=20mm,right=20mm,top=30mm,bottom=30mm]{geometry}
\usepackage{amsmath}
\usepackage{amsthm}
\usepackage{amssymb}
\usepackage{enumerate}
\usepackage{aliascnt}
\usepackage[bookmarks=false,
            pdftitle={Loewy lengths of centers of blocks},
            pdfauthor={Burkhard Külshammer and Benjamin Sambale},
            pdfkeywords={Loewy length, center of blocks},
            pdfstartview={FitH}]{hyperref}

\newtheorem{Theorem}{Theorem} 
\newaliascnt{Lemma}{Theorem}
\newtheorem{Lemma}[Lemma]{Lemma}
\aliascntresetthe{Lemma}
\newaliascnt{Proposition}{Theorem}
\newtheorem{Proposition}[Proposition]{Proposition}
\aliascntresetthe{Proposition}
\newaliascnt{Corollary}{Theorem}
\newtheorem{Corollary}[Corollary]{Corollary}
\aliascntresetthe{Corollary}

\numberwithin{equation}{section}

\allowdisplaybreaks[1]

\renewcommand{\phi}{\varphi}
\newcommand{\C}{\operatorname{C}}
\newcommand{\N}{\operatorname{N}}
\newcommand{\Z}{\operatorname{Z}}
\newcommand{\cohom}{\operatorname{H}}
\newcommand{\Aut}{\operatorname{Aut}}

\newcommand{\GL}{\operatorname{GL}}
\newcommand{\SL}{\operatorname{SL}}

\newcommand{\PSL}{\operatorname{PSL}}
\newcommand{\PGL}{\operatorname{PGL}}

\newcommand{\Br}{\operatorname{Br}}

\newcommand{\Ker}{\operatorname{Ker}}

\mathchardef\ordinarycolon\mathcode`\:  
 \mathcode`\:=\string"8000
 \begingroup \catcode`\:=\active
   \gdef:{\mathrel{\mathop\ordinarycolon}}
 \endgroup

\title{Loewy lengths of centers of blocks}
\author{Burkhard Külshammer\footnote{Institut für Mathematik, Friedrich-Schiller-Universität, 07743 Jena, Germany, 
\href{mailto:kuelshammer@uni-jena.de}{kuelshammer@uni-jena.de}} \ and Benjamin Sambale\footnote{Fachbereich Mathematik, TU Kaiserslautern, 67653 Kaiserslautern, Germany, 
\href{mailto:sambale@mathematik.uni-kl.de}{sambale@mathematik.uni-kl.de}}}
\date{\today}

\begin{document}
\frenchspacing
\maketitle
\begin{abstract}\noindent
Let $B$ be a block of a finite group with respect to an algebraically closed field $F$ of characteristic $p>0$. 
In a recent paper, Otokita gave an upper bound for the Loewy length $LL(ZB)$ of the center $ZB$ of $B$ in terms of a defect group $D$ of $B$. We refine his methods in order to prove the optimal bound $LL(ZB)\le LL(FD)$ whenever $D$ is abelian. We also improve Otokita's bound for non-abelian defect groups. As an application we classify the blocks $B$ such that $LL(ZB)\ge |D|/2$.
\end{abstract}


\section{Introduction}

We consider a block (algebra) $B$ of $FG$ where $G$ is a finite group and $F$ is an algebraically closed field of characteristic $p>0$. In general, the structure of $B$ is quite complicated and can only be described in restrictive special cases (e.\,g. blocks of defect $0$). For this reason, we are content here with the study of the center $ZB$ of $B$. This is a \emph{local} $F$-algebra in the sense that the \emph{Jacobson radical} $JZB$ has codimension $1$. It is well-known that the dimension of $ZB$ itself equals the number $k(B)$ of irreducible complex characters in $B$. In particular, this dimension is locally bounded by a theorem of Brauer and Feit~\cite{BrauerFeit}. Moreover, the number $l(B)$ of irreducible Brauer characters in $B$ is given by the dimension of the \emph{Reynolds ideal} $RB:=ZB\cap SB$ where $SB$ is the \emph{socle} of $B$. It follows that the dimension of the quotient $ZB/RB$ is locally determined by Brauer's theory of subsections. Here a \emph{$B$-subsection} is a pair $(u,b)$ where $u\in G$ is a $p$-element and $b$ is a Brauer correspondent of $B$ in $\C_G(u)$. 

In order to give better descriptions of $ZB$ we introduce the \emph{Loewy length} $LL(A)$ of a finite-dimensional $F$-algebra $A$ as the smallest positive integer $l$ such that $(JA)^l=0$.
A result by Okuyama~\cite{OkuyamaLL} states that $LL(ZB)\le|D|$ where $|D|$ is the order of a defect group $D$ of $B$. In fact, there is an open conjecture by Brauer~\cite[Problem~20]{BrauerLectures} asserting that even $\dim ZB\le|D|$. 
In a previous paper~\cite{KKS} jointly with Shigeo Koshitani, we have shown conversely that $LL(B)$ is bounded from below in terms of $|D|$. There is no such bound for $LL(ZB)$, but again an open question by Brauer~\cite[Problem~21]{BrauerLectures} asks if $\dim ZB$ can be bounded from below in terms of $|D|$.
 
Recently, Okuyama's estimate has been improved by Otokita~\cite{Otokita}. More precisely, if $\exp(D)$ is the exponent of $D$, he proved that 
\begin{equation}\label{otokita}
LL(ZB)\le |D|-\frac{|D|}{\exp(D)}+1.
\end{equation}
The present note is inspired by Otokita's methods. Our first result gives a local bound on the Loewy length of $ZB/RB$. Since $(JZB)(RB)\subseteq (JB)(SB)=0$, we immediately obtain a bound for $LL(ZB)$. In our main theorem we apply this bound to blocks with abelian defect groups as follows.

\begin{Theorem}\label{thm}
Let $B$ be a block of $FG$ with abelian defect group $D$. Then $LL(ZB/RB)<LL(FD)$ and $LL(ZB)\le LL(FD)$. 
\end{Theorem}

If, in the situation of \autoref{thm}, $D$ has type $(p^{a_1},\ldots,p^{a_r})$, then $LL(FD)=p^{a_1}+\ldots+p^{a_r}-r+1$ as is well-known. For $p$-solvable groups $G$, the stronger assertion $LL(B)=LL(FD)$ holds (see \cite[Theorem~K]{Kpsolv}). Similarly, if $D$ is cyclic, one can show more precisely that
\[LL(ZB)=LL(ZB/RB)+1=\dim ZB/RB+1=\frac{|D|-1}{l(B)}+1\]
(see \cite[Corollary~2.6]{KKS}).
By Broué-Puig~\cite{BrouePuig}, \autoref{thm} is best possible for nilpotent blocks. We conjecture conversely that the inequality is strict for non-nilpotent blocks (cf. \autoref{cor} and \autoref{broue} below). 

Arguing inductively, we also improve Otokita's bound for blocks with non-abelian defect groups. More precisely, we show in \autoref{nonabel} that
\[LL(ZB)\le\frac{|D|}{p}+\frac{|D|}{p^2}-\frac{|D|}{p^3}\]
(see also \autoref{large2}). Extending Otokita's work again, we use our results to classify all blocks $B$ with $LL(ZB)\ge|D|/2$ in \autoref{largeLL}.

It seems that in the non-abelian defect case the inequality $LL(ZB)\le LL(FD)$ is still satisfied. This holds for example if $D\unlhd G$ (see \cite[proof of Lemma~2.4]{Otokita}). We support this observation by computing the Loewy lengths of the centers of some blocks with small defect.
Finally, we take the opportunity to improve \cite[Proposition~2.2]{Otokita} (see \autoref{cz}). To do so we recall that the \emph{inertial quotient} $I(B)$ of $B$ is the group $\N_G(D,b_D)/D\C_G(D)$ where $b_D$ is a Brauer correspondent of $B$ in $\C_G(D)$. By the Schur-Zassenhaus Theorem, $I(B)$ can be embedded in the automorphism group $\Aut(D)$. Then
\[FD^{I(B)}:=\{x\in FD:a^{-1}xa=x\text{ for }a\in I(B)\}\]
is the algebra of fixed points.
Moreover, for a subset $U\subseteq G$ we define $U^+:=\sum_{u\in U}u\in FG$. Then $RFG$ has an $F$-basis consisting of the sums $S^+$ where $S$ runs through the $p'$-sections of $G$ (see for example \cite{KBemerkungen1}). Note that the trivial $p'$-section is given by the set $G_p$ of $p$-elements of $G$.

\section{Abelian defect groups}

By the results mentioned in the introduction we may certainly restrict ourselves to blocks with positive defect.

\begin{Proposition}\label{bound}
Let $B$ be a block of $FG$ with defect group $D\ne 1$. 
Let $(u_1,b_1),\ldots,(u_k,b_k)$ be a set of representatives for the conjugacy classes of non-trivial $B$-subsections.
Then the map
\begin{align*}
ZB/RB&\to\bigoplus_{i=1}^k{Zb_i/Rb_i},\\
z+RB&\mapsto\sum_{i=1}^k{\Br_{\langle u_i\rangle}(z)1_{b_i}+Rb_i}
\end{align*}
is an embedding of $F$-algebras where $\Br_{\langle u_i\rangle}:ZFG\to ZF\C_G(u_i)$ denotes the Brauer homomorphism. In particular,
\[LL(ZB/RB)\le\max\{LL(Zb_i/Rb_i):i=1,\ldots,k\}.\]
\end{Proposition}
\begin{proof}
First we consider the whole group algebra $FG$ instead of $B$. For this, let $v_1,\ldots,v_r$ be a set of representatives for the conjugacy classes of $p$-elements of $G$. Let $z:=\sum_{g\in G}{\alpha_gg}\in ZFG$. Then $z$ is constant on the conjugacy classes of $G$. It follows that $z$ is constant on the $p'$-sections of $G$ if and only if $\Br_{\langle v_i\rangle}(z)=\sum_{g\in\C_G(v_i)}{\alpha_gg}$ is constant on the $p'$-sections of $\C_G(v_i)$ for $i=1,\ldots,r$. Therefore, the map $ZFG/RFG\to\bigoplus_{i=1}^r{ZF\C_G(v_i)/RF\C_G(v_i)}$, $z+RFG\mapsto\sum_{i=1}^r{\Br_{\langle v_i\rangle}(z)+RF\C_G(v_i)}$ is a well-defined embedding of $F$-algebras. Now the first claim follows easily by projecting onto $B$, i.\,e. replacing $z$ by $z1_B$. The last claim is an obvious consequence.
\end{proof}

\begin{proof}[Proof of \autoref{thm}.]
Let $D\cong C_{p^{a_1}}\times\ldots\times C_{p^{a_r}}$. It is well-known that 
\[LL(FD)=LL(FC_{p^{a_1}}\otimes\ldots\otimes FC_{p^{a_r}})=p^{a_1}+\ldots+p^{a_r}-r+1.\] 
Hence, it suffices to show that $LL(ZB/RB)\le p^{a_1}+\ldots+p^{a_r}-r$. 

We argue by induction on $r$. If $r=0$, then we have $D=1$, $ZB=RB$ and $LL(ZB/RB)=0$.
Thus, we may assume that $r\ge 1$. Let $I:=I(B)$ be the inertial quotient of $B$. In order to apply \autoref{bound}, we consider a $B$-subsection $(u,b)$ with $1\ne u\in D$. Then $b$ has defect group $D$ and inertial quotient $\C_I(u)$. Since $I$ is a $p'$-group, we have $D=Q\times[D,\C_I(u)]$ with $Q:=\C_D(\C_I(u))\ne 1$. 
Let $\beta$ be the Brauer correspondent of $b$ in $\C_G(Q)\subseteq\C_G(u)$. By Watanabe~\cite[Theorem~2]{Watanabe2}, the Brauer homomorphism $\Br_D$ induces an isomorphism between $Zb$ and $Z\beta$. Since the intersection of a $p'$-section of $G$ with $\C_G(D)$ is a union of $p'$-sections of $\C_G(D)$, it follows that $\Br_D(Rb)\subseteq R\beta$.
On the other hand, $\dim_F Rb=l(b)=l(\beta)=\dim_F R\beta$ by \cite[Theorem~1]{Watanabe1}. Thus, we obtain $Zb/Rb\cong Z\beta/R\beta$ and it suffices to show that
\[LL(Z\beta/R\beta)\le p^{a_1}+\ldots+p^{a_r}-r.\]
Let $\overline{\beta}$ be the unique block of $\C_G(Q)/Q$ dominated by $\beta$. By \cite[Theorem~7]{KOW} (see also \cite[Theorem~1.2]{Fan}), it follows that the source algebra of $\beta$ is isomorphic to a tensor product of $FQ$ and the source algebra of $\overline{\beta}$. Since $\beta$ is Morita equivalent to its source algebra, we may assume in the following that $\beta=FQ\otimes\overline{\beta}$. 
Let $Q\cong C_{p^{a_1}}\times\ldots\times C_{p^{a_s}}$ with $1\le s\le r$. Since the defect group $D/Q$ of $\overline{\beta}$ has rank $r-s<r$, induction implies that 
\[LL(Z\overline{\beta}/R\overline{\beta})\le p^{a_{s+1}}+\ldots+p^{a_r}-r+s=:l.\] 
In particular, $(JZ\overline{\beta})^l\subseteq R\overline{\beta}$. 
Since $Q$ is an abelian $p$-group, we have $RFQ=SFQ\cong F$. Consequently, $LL(FQ/RFQ)=p^{a_1}+\ldots+p^{a_s}-s=:l'$, i.\,e. $(JFQ)^{l'}\subseteq RFQ$. Moreover, $SFQ\otimes S\overline{\beta}\subseteq S(FQ\otimes\overline{\beta})$.
Hence,
\[RFQ\otimes R\overline{\beta}\subseteq Z(FQ\otimes\overline{\beta})\cap S(FQ\otimes\overline{\beta})=R(FQ\otimes \overline{\beta})=R\beta.\]
Since $JZ\beta=J(FQ\otimes Z\overline{\beta})=JFQ\otimes Z\overline{\beta}+FQ\otimes JZ\overline{\beta}$, we see that $(JZ\beta)^{l+l'}$ is a sum of terms of the form $(JFQ)^i\otimes(JZ\overline{\beta})^j$ with $i+j=l+l'$. If $i>l'$, then $(JFQ)^i=0$. Similarly, if $j>l$, then $(JZ\overline{\beta})^j=0$. It follows that 
\[(JZ\beta)^{l+l'}=(JFQ)^{l'}\otimes(JZ\overline{\beta})^l\subseteq RFQ\otimes R\overline{\beta}\subseteq R\beta.\]
This proves the theorem, because $l+l'=p^{a_1}+\ldots+p^{a_r}-r$.
\end{proof}

Our theorem shows that Otokita's bound \eqref{otokita} is only optimal for nilpotent blocks with cyclic defect groups or defect group $C_2\times C_2$ (see \cite[Corollary~3.1]{Otokita}). 

The next result strengthens \cite[Proposition~2.2]{Otokita}.

\begin{Proposition}\label{cz}
Let $B$ be a block of $FG$ with defect group $D$. Moreover, let $c:=\dim F\Z(D)^{I(B)}$ and $z:=LL(F\Z(D)^{I(B)})$. Then $LL(ZB/RB)\le k(B)-l(B)+z-c$ and in particular 
\[LL(ZB)\le k(B)-l(B)+z-c+1.\]
\end{Proposition}
\begin{proof}
Let $K:=\Ker(\Br_D)\cap ZB\unlhd ZB$. Since $ZB$ is local, we have $K\subseteq JZB$. 
Furthermore, $RB+K/K$ annihilates the radical $JZB/K$ of $ZB/K$. It follows that $RB+K/K$ is contained in the socle of $ZB/K$. By Broué~\cite[Proposition~(III)1.1]{BroueBrauercoeff}, it is known that $\Br_D$ induces an isomorphism between $ZB/K$ and the symmetric $F$-algebra $F\Z(D)^{I(B)}$. The socle of the latter algebra has dimension $1$. Hence,
\[\dim RB+K/K\le 1.\]
On the other hand, $G_p^+\in RFG$. Therefore, $1_BG_p^+\in RB$ and
\[\Br_D(1_BG_p^+)=\Br_D(1_B)\Br_D(G_p^+)=\Br_D(1_B)\C_G(D)_p^+.\]
Here, $\Br_D(1_B)$ is the block idempotent of $b_D^{\N_G(D)}$ where $b_D$ is a Brauer correspondent of $B$ in $\C_G(D)$. In particular, $1_{b_D}\Br_D(1_B)=1_{b_D}$ and
\[0\ne 1_{b_D}\C_G(D)_p^+=1_{b_D}\Br_D(1_B)\C_G(D)_p^+=1_{b_D}\Br_D(1_BG_p^+).\]
From that we obtain $1_BG_p^+\notin K$ and $\dim RB+K/K=1$. This implies $RB+K/K=S(ZB/K)$ and 
\[LL(ZB/RB+K)=z-1.\]

Now we consider the lower section of $ZB$. Here we have
\begin{align*}
\dim RB+K/RB&=\dim RB+K-\dim RB=1+\dim K-l(B)\\
&=1+\dim ZB-\dim ZB/K-l(B)=1+k(B)-c-l(B).
\end{align*}
The claim follows easily.
\end{proof}

The invariant $c$ in \autoref{cz} is just the number of orbits of $I(B)$ on $\Z(D)$. Moreover, if $D$ and $I(B)$ are given, the number $z$ can be calculated by computer. It happens frequently that $I(B)$ acts trivially on $\Z(D)$. In this case, $c=\lvert\Z(D)\rvert$ and $z$ is determined by the isomorphism type of $\Z(D)$ as explained earlier. In particular, $LL(ZB)\le k(B)-l(B)$ whenever $\Z(D)$ is non-cyclic. 
Now we give a general upper bound on $z$.

\begin{Lemma}\label{lemfixed}
Let $P$ be a finite abelian $p$-group, and let $I$ be a $p'$-subgroup of $\Aut(P)$. Then
\[LL(FP^I)\le LL(F\C_P(I))+\frac{LL(F[P,I])-1}{2}.\]
\end{Lemma}
\begin{proof}
Since $FP^I=F\C_P(I)\otimes F[P,I]^I$, we may assume that $\C_P(I)=1$.
It suffices to show that $JFP^I\subseteq (JFP)^2$. It is well-known that $JFP$ is the augmentation ideal of $FP$ and $JFP^I=JFP\cap FP^I$. In particular, $I$ acts naturally on $JFP$ and on $JFP/(JFP)^2$. We regard $P/\Phi(P)$ as a vector space over $\mathbb{F}_p$.
By \cite[Remark~VIII.2.11]{Huppert2} there exists an isomorphism of $\mathbb{F}_p$-spaces
\[\Gamma:JFP/(JFP)^2\to F\otimes_{\mathbb{F}_p} P/\Phi(P)\]
sending $1-x+(JFP)^2$ to $1\otimes x\Phi(P)$ for $x\in P$. After choosing a basis, it is easy to see that $\Gamma(w^\gamma)=\Gamma(w)^\gamma$ for $w\in JFP/(JFP)^2$ and $\gamma\in I$. Let $w\in JFP^I\subseteq JFP$. Then $\Gamma(w+(JFP)^2)$ is invariant under $I$. It follows that $\Gamma(w+(JFP)^2)$ is a linear combination of elements of the form $\lambda\otimes x$ where $\lambda\in F$ and $x\in\C_{P/\Phi(P)}(I)$. However, by hypothesis, $\C_{P/\Phi(P)}(I)=\C_P(I)\Phi(P)/\Phi(P)=\Phi(P)$ and therefore $\Gamma(w+(JFP)^2)=0$. This shows $w\in(JFP)^2$ as desired.
\end{proof}

We describe special case which extends \autoref{thm}. Here, the action of $I(B)$ on $D$ is called \emph{semiregular} if all orbits on $D\setminus\{1\}$ have length $|I(B)|$.

\begin{Corollary}\label{cor}
Let $B$ be a block of $FG$ with abelian defect group $D$ such that $I:=I(B)$ acts semiregularly on $[D,I]$. Then
\[LL(ZB)=LL(FD^I)=LL(ZF[D\rtimes I])\le LL(F\C_D(I))+\frac{LL(F[D,I])-1}{2}.\]
\end{Corollary}
\begin{proof}
Let $Q:=\C_D(I)$ and let $b$ be a Brauer correspondent of $B$ in $\C_G(Q)$. By \cite[Theorem~2]{Watanabe2}, $ZB\cong Zb$. Moreover, by \cite[Theorem~7]{KOW} we have $ZB\cong FQ\otimes Z\overline{b}$ where $\overline{b}$ is the block of $\C_G(Q)/Q$ dominated by $b$. As usual, $\overline{b}$ has defect group $D/Q\cong[D,I]$ and inertial quotient $I(\overline{b})\cong I(B)$. It follows that $LL(ZB)=LL(FQ)+LL(Z\overline{b})-1$.
On the other hand, $FD^I\cong FQ\otimes F[D,I]^I$ and $F[D\rtimes I]\cong FQ\otimes F[[D,I]\rtimes I]$. Hence, we may assume that $Q=1$ and $[D,I]=D\ne 1$.

Let $(u_1,b_1),\ldots,(u_k,b_k)$ be a set of representatives for the $G$-conjugacy classes of non-trivial $B$-subsections. 
Since $I$ acts semiregularly on $D$, every block $b_i$ has inertial quotient $I(b_i)\cong\C_I(u_i)=1$. Hence, $b_i$ is nilpotent and $l(b_i)=1$.
With the notation of \autoref{cz}, it follows that 
\[k(B)-l(B)=\sum_{i=1}^k{l(b_i)}=\frac{|D|-1}{|I|}=c-1\] 
and $LL(ZB)\le LL(FD^I)$. By the proof of \autoref{cz}, we also have the opposite inequality $LL(ZB)\ge LL(FD^I)$. 
It is easy to see that $ZF[D\rtimes I]=FD^I\oplus \Gamma$ where $\Gamma$ is the subspace spanned by the non-trivial $p'$-class sums of $D\rtimes I$. By hypothesis, every non-trivial $p'$-conjugacy class is a $p'$-section of $D\rtimes I$. Hence, we obtain $\Gamma\subseteq RF[D\rtimes I]$. The claim $LL(FD^I)=LL(ZF[D\rtimes I])$ follows. The last claim is a consequence of \autoref{lemfixed}.
\end{proof}

\autoref{cor} applies for instance whenever $I$ has prime order. For example, if $|I|=2$, we have equality \[LL(ZB)=LL(F\C_D(I))+\frac{LL(F[D,I])-1}{2}\] 
by \cite[Proposition~2.6]{KKS}.
However, in general for a block $B$ with abelian defect group $D$ it may happen that $LL(ZB)>LL(FD^I)$. An example is given by the principal $3$-block of $G=(C_3\times C_3)\rtimes SD_{16}$. Here $LL(ZB)=3$ and $\dim FD^I=2$. 

In the situation of \autoref{cor}, $I$ is a complement in the Frobenius group $[D,I]\rtimes I$. In particular, the Sylow subgroups of $I$ are cyclic or quaternion groups. It follows that $I$ has trivial Schur multiplier. By a result of the first author~\cite{Kuelshammer}, the Brauer correspondent of $B$ in $\N_G(D)$ is Morita equivalent to $F[D\rtimes I]$.
In this way we see that \autoref{cor} is in accordance with Broué's Abelian Defect Group Conjecture. Moreover, Alperin's Weight Conjecture predicts $l(B)=k(I)$ in this situation. 
By a result of the second author (see \cite[Lemma~9]{SambaleC4} and \cite[Theorem~5]{SambaleC3}), we also have 
\[\dim ZB\le\lvert\C_D(I)\rvert\biggl(\frac{|[D,I]|-1}{l(B)}+l(B)\biggr)\le |D|.\]
Further properties of this class of blocks have been obtained in Kessar-Linckelmann~\cite{FrobeniusInertial}.
Nevertheless, it seems difficult to express $LL(FD^I)$ explicitly in terms of $D$ and $I$. Some special cases have been considered in \cite[Section~6.3]{Schwabrow}. 

Our next aim concerns the sharpness of \autoref{thm}. For this we need to discuss twisted group algebras of the form $F_\alpha[D\rtimes I(B)]$.

\begin{Lemma}\label{propuntw}
Let $P$ be a finite abelian $p$-group, and let $I$ be a non-trivial $p'$-subgroup of $\Aut(P)$. Then 
\[LL(ZF_\alpha[P\rtimes I])<LL(FP)\] 
for every $\alpha\in\cohom^2(I,F^\times)$.
\end{Lemma}
\begin{proof}
For the sake of brevity we write $PI$ instead of $P\rtimes I$. 
We may normalize $\alpha$ such that $x\cdot y$ in $F_\alpha[PI]$ equals $xy\in PI$ for all $x\in P$ and $y\in PI$.
By Passman~\cite[Theorem 1.6]{Passman}, 
\[JZF_\alpha[PI]=JF_\alpha[PI]\cap ZF_\alpha[PI]=(JFP\cdot F_\alpha[PI])\cap ZF_\alpha[PI].\]
It is known that $ZF_\alpha[PI]$ has a basis consisting of the $\alpha$-regular class sums (see for example \cite[Remark~4 on p. 155]{Conlon}). Hence, let $K$ be an $\alpha$-regular conjugacy class of $PI$. If $K\subseteq P$, then clearly $|K|1-K^+\in ZF_\alpha[PI]\cap JFP\subseteq JZF_\alpha[PI]$, since $JFP$ is the augmentation ideal of $FP$. Now assume that $K\subseteq PI\setminus P$ and $x\in K$. Then the $P$-orbit of $x$ (under conjugation) is the coset $x[x,P]$. Hence, $K$ is a disjoint union of cosets $x_1[x_1,P],\ldots,x_m[x_m,P]$. Since $I$ acts faithfully on $P$, we have $[x_i,P]\ne 1$ and $[x_i,P]^+\in JFP$ for $i=1,\ldots,m$.
It follows that $K^+\in (JFP\cdot F_\alpha[PI])\cap ZF_\alpha[PI]=JZF_\alpha[PI]$. In this way we obtain an $F$-basis of $JZF_\alpha[PI]$. Let $l:=LL(ZF_\alpha[PI])$. Then there exist conjugacy classes $K_1,\ldots,K_s\subseteq P$ and elements $x_1,\ldots,x_t\in PI\setminus P$ such that $s+t=l-1$ and
\[(|K_1|1-K_1^+)\ldots(|K_s|1-K_s^+)x_1[x_1,P]^+\ldots x_t[x_t,P]^+\ne 0\]
in $F_\alpha[PI]$. Since $x_i[x_i,P]=[x_i,P]x_i$, we conclude that
\begin{equation}\label{JFP}
0\ne(|K_1|1-K_1^+)\ldots(|K_s|1-K_s^+)[x_1,P]^+\ldots [x_t,P]^+\in FP.
\end{equation}
At this point, $\alpha$ does not matter anymore and we may assume that $\alpha=1$ in the following.
Since $ZF[PI]=F\C_P(I)\otimes ZF[[P,I]\rtimes I]$ and $FP=F\C_P(I)\otimes F[P,I]$, we may assume that $\C_P(I)=1$.
By \autoref{lemfixed} we have
\[s\le LL(FP^I)-1\le\frac{LL(FP)-1}{2}<LL(FP)-1.\]
Thus, we may assume that $t>0$. 
Since $x_1$ acts non-trivially on $[x_1,P]$, we obtain $|[x_1,P]|\ge 3$ and $[x_1,P]^+\in(JF[x_1,P])^2\subseteq (JFP)^2$. Also, $|K_i|1-K_i^+\in JFP$ for $i=1,\ldots,s$. Therefore, \eqref{JFP} shows that $(JFP)^l\ne 0$ and the claim follows.
\end{proof}

\begin{Proposition}\label{broue}
Let $B$ be a block of $FG$ with abelian defect group $D$. Suppose that the character-theoretic version of Broué's Conjecture holds for $B$. Then $LL(ZB)=LL(FD)$ if and only if $B$ is nilpotent.
\end{Proposition}
\begin{proof}
A nilpotent block $B$ satisfies $LL(ZB)=LL(FD)$ by Broué-Puig~\cite{BrouePuig}. Thus, we may assume conversely that $LL(ZB)=LL(FD)$. 
Broué's Conjecture implies $ZB\cong Zb$ where $b$ is the Brauer correspondent of $B$ in $\N_G(D)$. By Külshammer~\cite{Kuelshammer}, $Zb\cong ZF_\alpha[D\rtimes I(B)]$ for some $\alpha\in\cohom^2(I(B),F^\times)$. Now \autoref{propuntw} show that $I(B)=1$. Hence, $B$ must be nilpotent.
\end{proof}

\section{Non-abelian defect groups}

We start with a result about nilpotent blocks which might be of independent interest.

\begin{Proposition}\label{nilpotent}
For a non-abelian $p$-group $P$ we have $JZFP\subseteq JF[P'\Z(P)]\cdot FP$ and 
\[LL(ZFP)\le LL(FP'\Z(P))<LL(FP).\]
\end{Proposition}
\begin{proof}
We have already used that $JFP$ is the augmentation ideal of $FP$ and $JZFP=ZFP\cap JFP$. Hence, $JZFP$ is generated as an $F$-space by the elements $1-z$ and $K^+$ where $z\in\Z(P)$ and $K\subseteq P\setminus\Z(P)$ is a conjugacy class. 
Each such $K$ has the form $K=xU$ with $x\in P$ and $U\subseteq P'$. Since $|U|=|K|$ is a multiple of $p$, we have $U^+\in JFP'$. 
On the other hand, $1-z\in JF\Z(P)$ for $z\in\Z(P)$.
Setting $N:=P'\Z(P)$ we obtain $JZFP\subseteq FP\cdot JFN$. Since $P$ acts on $FN$ preserving the augmentation, we also have $FP\cdot JFN=JFN\cdot FP$. This shows $LL(ZFP)\le LL(FN)$. 

For the second inequality, note that $N\le \Z(P)\Phi(P)<P$. Hence, $FN^+=(JFN)^{LL(FN)-1}\subseteq(JFP)^{LL(FN)-1}$ and $(JFP)^{LL(FP)-1}=FP^+\ne FN^+$. Therefore, we must have $LL(FN)<LL(FP)$.
\end{proof}

If $P$ has class $2$, we have $P'\le\Z(P)$ and $JF\Z(P)\subseteq JZFP$. Hence, \autoref{nilpotent} implies $LL(ZFP)=LL(F\Z(P))$ in this case.

In the following we improve \eqref{otokita} for non-abelian defect groups. 
We make use of Otokita's inductive method:
\begin{equation}\label{otometh}
LL(ZB)\le\max\bigl\{(|\langle u\rangle|-1)LL(Z\overline{b}):(u,b)\text{ $B$-subsection}\bigr\}+1
\end{equation}
(see \cite[proof of Theorem~1.3]{Otokita}). Here $\overline{b}$ denotes the block of $\C_G(u)/\langle u\rangle$ dominated by $b$. By \cite[Lemma~1.34]{habil}, we may assume that $\overline{b}$ has defect group $\C_D(u)/\langle u\rangle$ where $D$ is a defect group of $B$.

We start with a detailed analysis of the defect groups of large exponent. 

\begin{Lemma}\label{lemtypes}
Let $P$ be a $p$-group such that $\Z(P)$ is cyclic and $|P:\Z(P)|=p^2$. Then one of the following holds:
\begin{enumerate}[(i)]
\item $P\cong\langle x,y\mid x^{p^{d-1}}=y^p=1,\ y^{-1}xy=x^{1+p^{d-2}}\rangle=:M_{p^d}$ for some $d\ge 3$.
\item $P\cong\langle x,y,z\mid x^{p^{d-2}}=y^p=z^p=[x,y]=[x,z]=1,\ [y,z]=x^{p^{d-3}}\rangle=:W_{p^d}$ for some $d\ge 3$.
\item $P\cong Q_8$.
\end{enumerate}
\end{Lemma}
\begin{proof}
Let $|P|=p^d$ with $d\ge 3$. If $\exp(P)=p^{d-1}$, then the result is well-known. Thus, we may assume that $\exp(P)=p^{d-2}$. 
Let $\Z(P)=\langle x\rangle$ and $D=\langle x,y,z\rangle$. Since $\langle x,y\rangle\cong\langle x,z\rangle\cong C_{p^{d-2}}\times C_p$, we may assume that $y^p=z^p=1$. Since $P$ is non-abelian, we have $1\ne[y,z]\in P'\le\Z(P)$. In particular, $P$ has nilpotency class $2$. It follows that $[y,z]^p=[y^p,z]=1$ and therefore $[y,z]=x^{p^{d-3}}$. Consequently, the isomorphism type of $P$ is uniquely determined. Conversely, one can construct such a group as a central product of $C_{p^{d-2}}$ and an extraspecial group of order $p^3$.
\end{proof}

\begin{Proposition}\label{p3-}
Let $B$ be a block of $FG$ with defect group $D\cong M_{p^d}$ or $D\cong Q_8$. Then one of the following holds:
\begin{enumerate}[(i)]
\item \[LL(ZB)=\frac{p^{d-2}-1}{l(B)}+1\le p^{d-2}=LL(ZFD)\le LL(FD).\]
\item $|D|=8$ and $LL(ZB)\le 3$.
\end{enumerate}
\end{Proposition} 
\begin{proof}
Suppose first that $p=2$. If $|D|=8$, then there are in total five possible fusion systems for $B$ and none of them is exotic (see \cite[Theorem~8.1]{habil}). 
By \cite{Cabanes}, the fusion system of $B$ determines the perfect isometry class of $B$. Since perfect isometries preserve the isomorphism type of $ZB$, we may assume that $B$ is the principal block of $FH$ where $H\in\{D_8,Q_8,S_4,\SL(2,3),\GL(3,2)\}$. 
A computation with GAP~\cite{GAP48} reveals that $LL(ZB)\le 3$ in all cases. Note that we may work over the field with two elements, since the natural structure constants of $ZFH$ (and of $ZB$) lie in the prime field of $F$. (The fusion system corresponding to $H=\GL(3,2)$ can be handled alternatively with \autoref{cz}.)
If $D\cong M_{2^d}$ with $d\ge 4$, then $B$ is nilpotent (\cite[Theorem~8.1]{habil}) and the result follows from the remark after \autoref{nilpotent}.

Now assume that $p>2$. By \cite{WatanabeMNA}, $B$ is perfectly isometric to its Brauer correspondent in $\N_G(D)$. Hence, we may assume that $D\unlhd G$. It is known that $B$ has cyclic inertial quotient $I(B)$ of order dividing $p-1$ (see \cite[proof of Theorem~8.8]{habil}). Hence, by \cite{Kuelshammer} we may assume that $G=D\rtimes I(B)$. Then $G$ has only one block and $ZB=ZFG$. Moreover, $l(B)=|I(B)|$. 
After conjugation, we may assume that $I(B)=\langle a\rangle$ acts non-trivially on $\langle x\rangle$ and trivially on $\langle y\rangle$. Since $|D'|=p$, the conjugacy classes of $D$ are either singletons in $\Z(D)$ or cosets of $D'$. Some of these classes are fused in $G$. The classes in $G\setminus D$ are cosets of $\langle x\rangle$. 
As usual, $ZFG$ is generated by the class sums and $JZFG$ is the augmentation ideal (intersected with $ZFG$). In particular, $JZFG$ contains the class sums of conjugacy classes whose length is divisible by $p$. Let $U_1,\ldots,U_k$ be the non-trivial orbits of $I(B)$ on $\Z(D)$. Then $JZFG$ also contains the sums $l(B)1_G-U_i^+$ for $i=1,\ldots,k$. For $u,v\in D$ we have
\begin{align*}
u(D')^+\cdot v(D')^+&=uv((D')^+)^2=0,\\
u(D')^+\cdot v\langle x\rangle^+&=uv(D')^+\langle x\rangle^+=0,\\
u\langle x\rangle^+\cdot v\langle x\rangle^+&=uv(\langle x\rangle^+)^2=0,\\
u(D')^+\cdot(l(B)1_G-U_i^+)&=l(B)u(D')^+-l(B)u(D')^+=0,\\
u\langle x\rangle^+\cdot (l(B)1_G-U_i^+)&=l(B)u\langle x\rangle^+-l(B)u\langle x\rangle^+=0.
\end{align*}
It follows that $(JZFG)^2=(JZF\langle x,a\rangle)^2$. Now the claim can be shown with \cite[Corollary~2.8]{KKS}.
\end{proof}

\begin{Lemma}\label{p3+}
Let $B$ be a block of $FG$ with defect group $D\cong W_{p^d}$. Then $LL(ZB)\le p^{d-1}-p+1$.
\end{Lemma}
\begin{proof}
If $|D|=8$, then the claim holds by \autoref{p3-}. Hence, we may exclude this case in the following.
We consider $B$-subsections $(u,b)$ with $1\ne u\in D$. As usual, we may assume that $b$ has defect group $\C_D(u)$. 

Suppose first that $I(B)$ acts faithfully on $\Z(D)$. We apply \autoref{bound}. 
If $u\notin\Z(D)$, then $\C_D(u)\cong C_{p^{d-2}}\times C_p$. Thus, \autoref{thm} implies $LL(Zb/Rb)\le p^{d-2}+p-2$. 
Now assume that $u\in\Z(D)$. The centric subgroups in the fusion system of $b$ are maximal subgroups of $D$. In particular, they are abelian of rank $2$. Now by \cite[Proposition~6.11]{habil}, it follows that $b$ is a controlled block. Since $I(b)\cong\C_{I(B)}(u)=1$, $b$ is nilpotent and $Zb\cong ZFD$. 
By \autoref{nilpotent} we obtain $LL(Zb/Rb)\le LL(Zb)=LL(ZFD)\le LL(F\Z(D))=p^{d-2}$.
Hence, \autoref{bound} gives
\[LL(ZB)\le LL(ZB/RB)+1\le p^{d-2}+p-1\le p^{d-1}-p+1.\] 

Now we deal with the case where $I(B)$ is non-faithful on $\Z(D)$. We make use of \eqref{otometh}. Let $|\langle u\rangle|=p^s$. The dominated block $\overline{b}$ has defect group $\C_D(u)/\langle u\rangle$. If $u\notin\Z(D)$, then $\lvert\C_D(u)/\langle u\rangle\rvert=p^{d-s-1}\ge p$ and
\[(p^s-1)LL(Z\overline{b})\le(p^s-1)p^{d-s-1}\le p^{d-1}-p.\] 
Next suppose that $u\in\Z(D)$. Then $D'\subseteq\langle u\rangle$ and $\overline{b}$ has defect group $D/\langle u\rangle\cong C_{p^{d-s-2}}\times C_p\times C_p$. In case $\langle u\rangle<\Z(D)$, we have $s\le d-3$ and \autoref{thm} implies
\[(p^s-1)LL(Z\overline{b})\le(p^s-1)(p^{d-s-2}+2p-2)\le p^{d-2}+2p^{d-2}-2p^{d-3}-3p+2\le p^{d-1}-p.\]
Finally, assume that $\langle u\rangle=\Z(D)$.  
By \cite[Lemma~3]{SambaleC4}, we have
\[I(\overline{b})\cong I(b)\cong\C_{I(B)}(u)\ne 1.\]
We want to show that $I(\overline{b})$ acts semiregularly on $D/\Z(D)$. 
Let $D=\langle x,y,z\rangle$ as in \autoref{lemtypes}, and let $\gamma\in I(\overline{b})$. Then $y^\gamma\equiv y^iz^j\pmod{\Z(D)}$ and $z^\gamma\equiv y^kz^l\pmod{\Z(D)}$ for some $i,j,k,l\in\mathbb{Z}$. Since $D$ has nilpotency class $2$, we have
\[[y,z]=[y,z]^\gamma=[y^\gamma,z^\gamma]=[y^iz^j,y^kz^l]=[y,z]^{il-jk}.\]
It follows that $il-jk\equiv 1\pmod{p}$ and $I(\overline{b})\le\SL(2,p)$. As a $p'$-subgroup of $\SL(2,p)$, $I(\overline{b})$ acts indeed semiregularly on $D/\Z(D)$. Thus, \autoref{cor} shows that 
\[(p^s-1)LL(Z\overline{b})\le (p^{d-2}-1)p=p^{d-1}-p\]
Therefore, the claim follows from \eqref{otometh}.
\end{proof}

We do not expect that \autoref{p3+} is sharp. In fact, Jennings's Theorem~\cite{Jennings} shows that $LL(FW_{p^3})=4p-3$. Even in this small case the perfect isometry classes are not known (see for example \cite{ExtraspecialExpp}). 

We are now in a position to deal with all non-abelian defect groups.

\begin{Theorem}\label{nonabel}
Let $B$ be a block of $FG$ with non-abelian defect group of order $p^d$. Then 
\[LL(ZB)\le p^{d-1}+p^{d-2}-p^{d-3}.\]
\end{Theorem}
\begin{proof}
We argue by induction on $d$. Let $D$ be a defect group of $B$. 
Again we will use \eqref{otometh}. Let $(u,b)$ be a $B$-subsection with $u\in D$ of order $p^s\ne 1$. As before, we may assume that the dominated block $\overline{b}$ has defect group $\C_D(u)/\langle u\rangle$. 
If $\C_D(u)/\langle u\rangle$ is cyclic, then $\C_D(u)$ is abelian and therefore $\C_D(u)<D$. Hence, 
\[(p^s-1)LL(Z\overline{b})\le(p^s-1)p^{d-s-1}\le p^{d-1}-1\le p^{d-1}+p^{d-2}-p^{d-3}-1.\]

Suppose next that $\C_D(u)/\langle u\rangle$ is abelian of type $(p^{a_1},\ldots,p^{a_r})$ with $r\ge 2$. 
If $s=d-2$, then $D$ fulfills the assumption of \autoref{lemtypes}. Hence, by \autoref{p3-} and \autoref{p3+} we conclude that \[LL(ZB)\le p^{d-1}-p+1\le p^{d-1}+p^{d-2}-p^{d-3}.\] 
Consequently, we can restrict ourselves to the case $s\le d-3$.
\autoref{thm} shows that 
\[
LL(Z\overline{b})\le p^{a_1}+\ldots+p^{a_r}-r+1\le p^{a_1+\ldots+a_{r-1}}+p^{a_r}-1\le\frac{\lvert\C_D(u)\rvert}{p^{s+1}}+p-1.
\]
Hence, one gets
\[(p^s-1)LL(Z\overline{b})\le(p^s-1)(p^{d-s-1}+p-1)\le p^{d-1}+p^{s+1}-p^s-1\le p^{d-1}+p^{d-2}-p^{d-3}-1.\]

It remains to consider the case where $\C_D(u)/\langle u\rangle$ is non-abelian. Here induction gives
\[(p^s-1)LL(Z\overline{b})\le(p^s-1)(p^{d-s-1}+p^{d-s-2}-p^{d-s-3})\le p^{d-1}+p^{d-2}-p^{d-3}-1.\]
Now the claim follows with \eqref{otometh}.
\end{proof}

Doing the analysis in the proof more carefully, our bound can be slightly improved, but this does not affect the order of magnitude. Note also that \autoref{nonabel} improves Eq.~\eqref{otokita} even in case $p=2$, because then $\exp(D)\ge 4$. Nevertheless, we develop a stronger bound for $p=2$ in the following.
We begin with the $2$-blocks of defect $4$. The definition of the minimal non-abelian group $MNA(2,1)$ can be found in \cite[Theorem~12.2]{habil}. 
The following proposition covers all non-abelian $2$-groups of order $16$. 

\begin{Proposition}\label{d4}
Let $B$ be a block of $FG$ with defect group $D$. Then 
\[LL(ZB)\le\begin{cases}
3&\text{if }D\cong C_4\rtimes C_4,\\
4&\text{if }D\in\{M_{16},\,D_8\times C_2,\,Q_8\times C_2,\,MNA(2,1)\},\\
5&\text{if }D\in\{D_{16},\,Q_{16},\,SD_{16},\,W_{16}\}.
\end{cases}\]
In all cases we have $LL(ZB)\le LL(FD)$.
\end{Proposition}
\begin{proof}
The case $D\cong M_{16}$ has already been done in \autoref{p3-}.
For the metacyclic group $D\cong C_4\rtimes C_4$, $B$ is nilpotent (see \cite[Theorem~8.1]{habil}) and the result follows from \autoref{nilpotent}. For the dihedral, quaternion, semidihedral and minimal non-abelian groups the perfect isometry class is uniquely determined by the fusion system of $B$ (see \cite{Cabanes,Sambalemna3}). Moreover, all these fusion systems are non-exotic (see \cite[Theorem~10.17]{habil}). In particular, $LL(ZB)\le LL(ZFH)$ for some finite group $H$. More precisely, if $B$ is non-nilpotent, we may consider the following groups $H$:
\begin{itemize}
\item $\PGL(2,7)$ and $\PSL(2,17)$ if $D\cong D_{16}$,
\item $\SL(2,7)$ and $\texttt{SmallGroup}(240,89)\cong 2.S_5$ if $D\cong Q_{16}$,
\item $M_{10}$ (Mathieu group), $\GL(2,3)$ and $\PSL(3,3)$ if $D\cong SD_{16}$,
\item $\texttt{SmallGroup}(48,30)\cong A_4\rtimes C_4$ if $D\cong MNA(2,1)$.
\end{itemize}
For all these groups $H$ the number $LL(ZFH)$ can be determined with GAP~\cite{GAP48}. 

Finally, for $D\in\{D_8\times C_2,\,Q_8\times C_2,\,W_{16}\}$ one can enumerate the possible generalized decomposition matrices of $B$ up to basic sets (see \cite[Propositions~3, 4 and 5]{SambalekB2}). In each case the isomorphism type of $ZB$ can be determined with a result of Puig~\cite{Puigcenter}. We omit the details. Observe that we improve \autoref{p3+} for $D\cong W_{16}$. Finally, the claim $LL(ZB)\le LL(FD)$ can be shown with Jennings's Theorem~\cite{Jennings} or one consults \cite[Corollary~4.2.4 and Table~4.2.6]{KarpilovskyJ}.
\end{proof}

Next we elaborate on \autoref{lemtypes}.

\begin{Lemma}\label{lem}
Let $B$ be a $2$-block of $FG$ with non-abelian defect group $D$ such that there exists a $z\in\Z(D)$ with $D/\langle z\rangle\cong C_{2^n}\times C_2$ where $n\ge 2$. Then $LL(ZB)<|D|/2$.
\end{Lemma}
\begin{proof}
By hypothesis there exist two maximal subgroups $M_1$ and $M_2$ of $D$ containing $z$ such that $M_1/\langle z\rangle\cong M_2/\langle z\rangle\cong C_{2^n}$. It follows that $M_1$ and $M_2$ are abelian. Since $D=M_1M_2$, we obtain $\Z(D)=M_1\cap M_2$ and $|D:\Z(D)|=4$. 
This implies $|D'|=2$ (see e.\,g. \cite[Lemma~1.1]{Berkovich1})). Obviously, $D'\le\langle z\rangle$. By \autoref{lemtypes} we may assume that $\Z(D)$ is abelian of rank $2$. 

Suppose for the moment that $B$ is nilpotent. 
Since $\Z(D)$ is not cyclic, $D\not\cong M_{2^m}$ for all $m$. Now a result of Koshitani-Motose~\cite[Theorems~4 and 5]{MotoseK} shows that
\[LL(ZB)=LL(ZFD)\le LL(FD)<\frac{|D|}{2}.\]

For the remainder of the proof we may assume that $B$ is not nilpotent.
Suppose that $\Z(D)=\Phi(D)$. Then $D$ is minimal non-abelian and it follows from \cite[Theorem~12.4]{habil} that $D\cong MNA(r,1)$ for some $r\ge 2$. By \autoref{d4} we can assume that $r\ge 3$.
By the main result of \cite{Sambalemna3}, $B$ is isotypic to the principal block of $H:=A_4\rtimes C_{2^r}$. In particular, $LL(ZB)\le LL(FH)$. Note that $H$ contains a normal subgroup $N\cong C_{2^{r-1}}\times C_2\times C_2$ such that $H/N\cong S_3$ (see \cite[Lemma~2]{Sambalemna3}). By Passman~\cite[Theorem~1.6]{PassmanTw}, $(JFH)^2\subseteq (JFN)(FH)=(FH)(JFN)$. It follows that 
\[LL(FH)\le 2LL(FN)=2(2^{r-1}+2)<2^{r+1}=\frac{|D|}{2}.\] 

Thus, we may assume $\lvert D:\Phi(D)\rvert=8$ in the following. 
Let $\mathcal{F}$ be the fusion system of $B$. Suppose that there exists an $\mathcal{F}$-essential subgroup $Q\le D$ (see \cite[Definition~6.1]{habil}). Then $z\in\Z(D)\le\C_D(Q)\le Q$ and $Q$ is abelian. Moreover, $|D:Q|=2$. 
It is well-known that $\Aut_{\mathcal{F}}(Q)$ acts faithfully on $Q/\Phi(Q)$ (see \cite[p. 64]{habil}).
Since $D/Q\le\Aut_{\mathcal{F}}(Q)$, we obtain $D'\nsubseteq\Phi(Q)$. On the other hand, $z^2\in\Phi(Q)$. This shows that $D'=\langle z\rangle$ and $D/D'$ has rank $2$. However, this contradicts $|D:\Phi(D)|=8$.

Therefore, $B$ is a controlled block and $\Aut(D)$ is not a $2$-group. Let $1\ne\alpha\in\Aut(D)$ be of odd order. 
Then $\alpha$ acts trivially on $D'$ and on $\Omega(\Z(D))/D'$, since $\Z(D)$ has rank $2$. Hence, $\alpha$ acts trivially on $\Omega(\Z(D))$ and also on $\Z(D)$. 
But then $\alpha$ acts non-trivially on $D/\langle z\rangle\cong C_{2^n}\times C_2$ which is impossible. This contradiction shows that there are no more blocks with the desired property.
\end{proof}

\begin{Proposition}\label{large2}
Let $B$ be a $2$-block of $FG$ with non-abelian defect group of order $2^d$. Then $LL(ZB)<2^{d-1}$.
\end{Proposition}
\begin{proof}
We mimic the proof of \autoref{nonabel}. Let $D$ be a defect group of $B$, and let $(u,b)$ be a $B$-subsection such that $u$ has order $2^s>1$. It suffices to show that $(2^s-1)LL(Z\overline{b})\le 2^{d-1}-2$. If $\C_D(u)/\langle u\rangle$ is cyclic, then $\C_D(u)$ is abelian and $\C_D(u)<D$. Then we obtain
\[(2^s-1)LL(Z\overline{b})\le(2^s-1)2^{d-s-1}=2^{d-1}-2^{d-s-1}.\]
We may assume that $s=d-1$. Then by \autoref{p3-}, we may assume that $D$ is dihedral, semidihedral or quaternion. Moreover, by \autoref{d4}, we may assume that $d\ge 5$. Then \cite[Theorem~8.1]{habil} implies 
\[LL(ZB)\le\dim ZB=k(B)\le2^{d-2}+5<2^{d-1}.\]

Now suppose that $\C_D(u)/\langle u\rangle$ is abelian of type $(2^{a_1},\ldots,2^{a_r})$ with $r\ge 2$. As in \autoref{nonabel}, we may assume that $s\le d-3$. If $a_1=1$ and $r=2$, then by \autoref{lem}, we may assume that $\C_D(u)<D$. Hence, we obtain 
\[(2^s-1)LL(Z\overline{b})\le (2^s-1)(2^{d-s-2}+1)\le 2^{d-2}+2^{d-3}\le 2^{d-1}-2\]
in this case.
Now suppose that $r\ge 3$ or $a_i>1$ for $i=1,2$. 
If $r=3$ and $a_1=a_2=a_3=1$, we have $(2^s-1)LL(Z\overline{b})\le 2^{d-1}-4$. In the remaining cases we have $s\le d-4$ and
\[(2^s-1)LL(Z\overline{b})\le (2^s-1)(2^{d-s-2}+3)\le 2^{d-2}+3\cdot 2^{d-4}\le 2^{d-1}-2.\]
Finally, suppose that $\C_D(u)/\langle u\rangle$ is non-abelian. Then the claim follows by induction on $d$.
\end{proof}

\begin{Corollary}\label{largeLL}
Let $B$ be a block of $FG$ with defect group $D$. Then $LL(ZB)\ge |D|/2$ if and only if one of the following holds:
\begin{enumerate}[(i)]
\item $D$ is cyclic and $l(B)\le 2$,
\item $D\cong C_{2^n}\times C_2$ for some $n\ge 1$,
\item $D\cong C_2\times C_2\times C_2$ and $B$ is nilpotent,
\item $D\cong C_3\times C_3$ and $B$ is nilpotent.
\end{enumerate}
\end{Corollary}
\begin{proof}
Suppose that $LL(ZB)\ge |D|/2$.
Then by \autoref{nonabel} and \autoref{large2}, $D$ is abelian. If $D$ is cyclic, we have $LL(ZB)=\frac{|D|-1}{l(B)}+1$. If additionally $l(B)\ge 3$, then we get the contradiction $|D|\le 4$. 
Now suppose that $D$ is not cyclic. Then
\[\frac{|D|}{2}\le LL(ZB)\le \frac{|D|}{p}+p-1\]
by \autoref{thm} and we conclude that 
\[p^2\le|D|\le\frac{2p(p-1)}{p-2}.\]
This yields $p\le 3$. Suppose first that $p=3$. Then we have $D\cong C_3\times C_3$ and $5=LL(ZB)\le k(B)-l(B)+1$ by \autoref{cz}. It follows from \cite{Kiyota} that $I(B)\notin\{C_4,C_8,Q_8,SD_{16}\}$ (note that $k(B)-l(B)$ is determined locally). The case $I(B)\cong C_2$ is excluded by \autoref{cor}. Hence, we may assume that $I(B)\in\{C_2\times C_2,D_8\}$. By \cite[Theorem~3]{SambaleBroue} and \cite[Lemma~2]{SambaleExc}, $B$ is isotypic to its Brauer correspondent in $\N_G(D)$. This gives the contradiction $LL(ZB)\le 3$. Therefore, $B$ must be nilpotent and $LL(ZB)=5$.

Now let $p=2$. Then $D$ has rank at most $3$. If the rank is $3$, we obtain $LL(ZB)\le 2^{d-2}+2$ and $d=3$. In this case, $B$ is nilpotent or $I(B)\cong C_7\rtimes C_3$ by \autoref{cor}. By \cite{KKL}, $B$ is isotypic to its Brauer correspondent in $\N_G(D)$. From that we can derive that $B$ is nilpotent and $LL(ZB)=4$.
It remains to handle defect groups of rank $2$. Here, $D\cong C_{2^n}\times C_2$ for some $n\ge 1$. If $n\ge 2$, then $B$ is always nilpotent and $LL(ZB)=2^n+1$. If $n=1$, then both possibilities $l(B)\in\{1,3\}$ give $LL(ZB)\ge 2$. 

Conversely, we have seen that all our examples actually satisfy $LL(ZB)\ge|D|/2$.
\end{proof}

The following approach gives more accurate results for a given arbitrary defect group.
For a finite $p$-group $P$ we define a recursive function $\mathcal{L}$ as follows:
\[\mathcal{L}(P):=\begin{cases}
p^{a_1}+\ldots+p^{a_r}-r+1&\text{if }P\cong C_{p^{a_1}}\times\ldots\times C_{p^{a_r}},\\
p^{d-2}&\text{if }P\cong M_{p^d}\text{ with }p^d\ne 8,\\
p^{d-1}-p+1&\text{if }P\cong W_{p^d}\text{ with }p^d\ne 16,\\
3&\text{if }P\in\{D_8,\,Q_8,\,C_4\rtimes C_4\},\\
4&\text{if }P\in\{D_8\times C_2,\,Q_8\times C_2,\,MNA(2,1)\},\\
5&\text{if }P\in\{D_{16},\,Q_{16},\,SD_{16},\,W_{16}\},\\
\max\bigl\{(|\langle u\rangle|-1)\mathcal{L}(\C_P(u)/\langle u\rangle):1\ne u\in P\bigr\}+1&\text{otherwise}.
\end{cases}\]
Then, by the results above, every block $B$ of $FG$ with defect group $D$ satisfies $LL(ZB)\le\mathcal{L}(D)$. For example, there are only three non-abelian defect groups of order $3^6$ giving the worst case estimate $LL(ZB)\le 287$.

In general, it is difficult to give good lower bounds on $LL(ZB)$ (cf. \cite[Corollary~2.7]{KKS}). Assume for instance that $\mathbb{F}_{p^n}$ is the field with $p^n$ elements and $G=\mathbb{F}_{p^n}\rtimes\mathbb{F}_{p^n}^\times$ for some $n\ge 1$. Then $G$ has only one block $B$ and $k(B)-l(B)=1$. It follows that $LL(ZB)=2$. In particular, the defect of $B$ is generally not bounded in terms of $LL(ZB)$. 

\section*{Acknowledgment}
The second author is supported by the German Research Foundation (project SA 2864/1-1) and the Daimler and Benz Foundation (project 32-08/13).


\begin{thebibliography}{10}

\bibitem{Berkovich1}
Y. Berkovich, \textit{Groups of prime power order. {V}ol. 1}, de Gruyter
  Expositions in Mathematics, Vol. 46, Walter de Gruyter GmbH \& Co. KG,
  Berlin, 2008.

\bibitem{BrauerLectures}
R. Brauer, \textit{Representations of finite groups}, in: Lectures on {M}odern
  {M}athematics, {V}ol. {I}, 133--175, Wiley, New York, 1963.

\bibitem{BrauerFeit}
R. Brauer and W. Feit, \textit{On the number of irreducible characters of
  finite groups in a given block}, Proc. Nat. Acad. Sci. U.S.A. \textbf{45}
  (1959), 361--365.

\bibitem{BroueBrauercoeff}
M. Broué, \textit{Brauer coefficients of {$p$}-subgroups associated with a
  {$p$}-block of a finite group}, J. Algebra \textbf{56} (1979), 365--383.

\bibitem{BrouePuig}
M. Broué and L. Puig, \textit{A {F}robenius theorem for blocks}, Invent. Math.
  \textbf{56} (1980), 117--128.

\bibitem{Cabanes}
M. Cabanes and C. Picaronny, \textit{Types of blocks with dihedral or
  quaternion defect groups}, J. Fac. Sci. Univ. Tokyo Sect. IA Math.
  \textbf{39} (1992), 141--161. Revised version:
  \url{http://www.math.jussieu.fr/~cabanes/type99.pdf}.

\bibitem{Conlon}
S.~B. Conlon, \textit{Twisted group algebras and their representations}, J.
  Austral. Math. Soc. \textbf{4} (1964), 152--173.

\bibitem{Fan}
Y. Fan, \textit{Relative local control and the block source algebras}, Sci.
  China Ser. A \textbf{40} (1997), 785--798.

\bibitem{GAP48}
The GAP~Group, \textit{GAP -- Groups, Algorithms, and Programming, Version
  4.8.4}; 2016, (\url{http://www.gap-system.org}).

\bibitem{Huppert2}
B. Huppert and N. Blackburn, \textit{Finite groups. {II}}, Grundlehren der
  Mathematischen Wissenschaften, Vol. 242, Springer-Verlag, Berlin, 1982.

\bibitem{Jennings}
S.~A. Jennings, \textit{The structure of the group ring of a {$p$}-group over a
  modular field}, Trans. Amer. Math. Soc. \textbf{50} (1941), 175--185.

\bibitem{KarpilovskyJ}
G. Karpilovsky, \textit{The {J}acobson radical of group algebras},
  North-Holland Mathematics Studies, Vol. 135, North-Holland Publishing Co.,
  Amsterdam, 1987.

\bibitem{KKL}
R. Kessar, S. Koshitani and M. Linckelmann, \textit{Conjectures of {A}lperin
  and {B}roué for {$2$}-blocks with elementary abelian defect groups of order
  {$8$}}, J. Reine Angew. Math. \textbf{671} (2012), 85--130.

\bibitem{FrobeniusInertial}
R. Kessar and M. Linckelmann, \textit{On blocks with {F}robenius inertial
  quotient}, J. Algebra \textbf{249} (2002), 127--146.

\bibitem{Kiyota}
M. Kiyota, \textit{On {$3$}-blocks with an elementary abelian defect group of
  order {$9$}}, J. Fac. Sci. Univ. Tokyo Sect. IA Math. \textbf{31} (1984),
  33--58.

\bibitem{KKS}
S. Koshitani, B. Külshammer and B. Sambale, \textit{On Loewy lengths of
  blocks}, Math. Proc. Cambridge Philos. Soc. \textbf{156} (2014), 555--570.

\bibitem{KBemerkungen1}
B. K{\"u}lshammer, \textit{Bemerkungen \"uber die {G}ruppenalgebra als
  symmetrische {A}lgebra}, J. Algebra \textbf{72} (1981), 1--7.

\bibitem{KOW}
B. K{\"u}lshammer, T. Okuyama and A. Watanabe, \textit{A lifting theorem with
  applications to blocks and source algebras}, J. Algebra \textbf{232} (2000),
  299--309.

\bibitem{Kpsolv}
B. Külshammer, \textit{On {$p$}-blocks of {$p$}-solvable groups}, Comm.
  Algebra \textbf{9} (1981), 1763--1785.

\bibitem{Kuelshammer}
B. Külshammer, \textit{Crossed products and blocks with normal defect groups},
  Comm. Algebra \textbf{13} (1985), 147--168.

\bibitem{MotoseK}
K. Motose, \textit{On a theorem of {S}. {K}oshitani}, Math. J. Okayama Univ.
  \textbf{20} (1978), 59--65.

\bibitem{OkuyamaLL}
T. Okuyama, \textit{On the radical of the center of a group algebra}, Hokkaido
  Math. J. \textbf{10} (1981), 406--408.

\bibitem{Otokita}
Y. Otokita, \textit{Some studies on Loewy lengths of centers of $p$-blocks},
  \href{http://arxiv.org/abs/1605.07949v2}{arXiv:1605.07949v2}.

\bibitem{Passman}
D.~S. Passman, \textit{{$p$}-solvable doubly transitive permutation groups},
  Pacific J. Math. \textbf{26} (1968), 555--577.

\bibitem{PassmanTw}
D.~S. Passman, \textit{Radicals of twisted group rings}, Proc. London Math.
  Soc. (3) \textbf{20} (1970), 409--437.

\bibitem{Puigcenter}
L. Puig, \textit{The center of a block}, in: Finite reductive groups ({L}uminy,
  1994), 361--372, Progr. Math., Vol. 141, Birkh\"auser, Boston, MA, 1997.

\bibitem{ExtraspecialExpp}
A. Ruiz and A. Viruel, \textit{The classification of {$p$}-local finite groups
  over the extraspecial group of order {$p^3$} and exponent {$p$}}, Math. Z.
  \textbf{248} (2004), 45--65.

\bibitem{SambaleExc}
B. Sambale, \textit{Isotypies for the quasisimple groups with exceptional Schur
  multiplier}, J. Algebra Appl. (to appear),
  \href{http://dx.doi.org/10.1142/S0219498817500785}{DOI:10.1142/S0219498817500785}.

\bibitem{SambalekB2}
B. Sambale, \textit{Cartan matrices and {B}rauer's {$k(B)$}-conjecture {II}},
  J. Algebra \textbf{337} (2011), 345--362.

\bibitem{habil}
B. Sambale, \textit{Blocks of finite groups and their invariants}, Springer
  Lecture Notes in Math., Vol. 2127, Springer-Verlag, Cham, 2014.

\bibitem{SambaleBroue}
B. Sambale, \textit{Broué's isotypy conjecture for the sporadic groups and
  their covers and automorphism groups}, Internat. J. Algebra Comput.
  \textbf{25} (2015), 951--976.

\bibitem{SambaleC3}
B. Sambale, \textit{Cartan matrices and {B}rauer's {$k(B)$}-{C}onjecture
  {III}}, Manuscripta Math. \textbf{146} (2015), 505--518.

\bibitem{SambaleC4}
B. Sambale, \textit{Cartan matrices and {B}rauer's {$k(B)$}-{C}onjecture {IV}},
  J. Math. Soc. Japan (to appear).

\bibitem{Sambalemna3}
B. Sambale, \textit{$2$-Blocks with minimal nonabelian defect groups III},
  Pacific J. Math. \textbf{280} (2016), 475--487.

\bibitem{Schwabrow}
I. Schwabrow, \textit{The centre of a block}, PhD thesis, Manchester, 2016.

\bibitem{WatanabeMNA}
A. Watanabe, \textit{Perfect isometries for blocks with metacyclic, minimal
  non-abelian defect groups}, manuscript.

\bibitem{Watanabe2}
A. Watanabe, \textit{Note on a {$p$}-block of a finite group with abelian
  defect group}, Osaka J. Math. \textbf{26} (1989), 829--836.

\bibitem{Watanabe1}
A. Watanabe, \textit{Notes on {$p$}-blocks of characters of finite groups}, J.
  Algebra \textbf{136} (1991), 109--116.

\end{thebibliography}
\end{document}